\documentclass[11pt]{amsart}
\usepackage{amsmath}
\usepackage{amsfonts}
\usepackage{amsthm}
\pdfoutput=1
\usepackage{stmaryrd}
\usepackage{amssymb}
\usepackage{graphicx}
\usepackage{subfig}
\usepackage{amsrefs}
\usepackage{amsmath,ifthen, amsfonts, amssymb,
srcltx, amsopn, color, enumerate, mathrsfs, overpic, lineno}
\usepackage[cmtip,arrow]{xy}
\usepackage{pb-diagram, pb-xy}
\usepackage{overpic}
\usepackage{subfig}

\DeclareMathOperator{\DM}{F}

\usepackage[T1]{fontenc}
\usepackage{libertine}
\let\oldmarginpar\marginpar
\renewcommand\marginpar[1]{\-\oldmarginpar[\raggedleft\footnotesize #1]
{\raggedright\footnotesize #1}}

\begin{document}
\newtheorem{theorem}{Theorem}[section]
\newtheorem{lemma}[theorem]{Lemma}
\newtheorem{corollary}[theorem]{Corollary}

\newcounter{intronum}

\newtheorem*{thma}{Theorem 3.3}
\newtheorem*{thmb}{Theorem 4.3}
\newtheorem*{corollaryc}{Corollary 4.4}
\newtheorem*{lemmad}{Lemma 5.1}
\newtheorem*{corollarye}{Corollary 5.2}
\newtheorem{thm}{Theorem}[section]
\newtheorem{lem}[thm]{Lemma}
\newtheorem{Slem}[thm]{Speculative Lemma}
\newtheorem{mainlem}[thm]{Main Lemma}
\newtheorem{cor}[thm]{Corollary}
\newtheorem{conj}[thm]{Conjecture}
\newtheorem{conji}[intronum]{Conjecture}
\newtheorem{prop}[thm]{Proposition}
\newtheorem*{quasithm}{Quasi-Theorem}
\newtheorem{thmi}{Theorem}
\newtheorem{cori}[thmi]{Corollary}
\newtheorem*{lemn}{Lemma}

\theoremstyle{definition}
\newtheorem{definition}[thm]{Definition}
\newtheorem{defni}{Definition}
\newtheorem{rem}[thm]{Remark}
\newtheorem{remi}[intronum]{Remark}
\newtheorem{exmp}[thm]{Example}
\newtheorem{exmpi}[intronum]{Example}
\newtheorem{warning}[thm]{Warning}
\newtheorem{history}[thm]{Historical Note}
\newtheorem{notation}[thm]{Notation}
\newtheorem{conv}[thm]{Convention}
\newtheorem{lede}[thm]{Lemma-Definition}
\newtheorem{openprob}[thm]{Open Problem}
\newtheorem{prob}[thm]{Problem}
\newtheorem{question}[thm]{Question}
\newtheorem{claim}{Claim}
\newtheorem*{staI}{Statement I}
\newtheorem*{staII}{Statement II}
\newtheorem{proc}{Procedure}
\newtheorem{claim*}{Claim}
\newtheorem*{obs}{Observation}
\newtheorem{cons}[thm]{Construction}




\newsavebox{\commentbox}
\newenvironment{com}%
{\ifthenelse{\equal{\showcomments}{yes}}%
{\footnotemark
        \begin{lrbox}{\commentbox}
        \begin{minipage}[t]{1.25in}\raggedright\sffamily\tiny
        \footnotemark[\arabic{footnote}]}
{\begin{lrbox}{\commentbox}}}%
{\ifthenelse{\equal{\showcomments}{yes}}%
{\end{minipage}\end{lrbox}\marginpar{\usebox{\commentbox}}}
{\end{lrbox}}}

\newcommand{\visual}{\partial}

\title{On the Residual Finiteness Growths of Particular Hyperbolic Manifold Groups}
\author[P. Patel]{Priyam Patel}
\address{Department of Mathematics, Purdue University, West Lafayette, Indiana, USA}
\email{patel376@purdue.edu}
\date{\today}

\begin{abstract}
We give a quantification of residual finiteness for the fundamental groups of hyperbolic manifolds that admit a totally geodesic immersion to a compact, right-angled Coxeter orbifold of dimension 3 or 4. Specifically, we give explicit upper bounds on residual finiteness that are linear in terms of geodesic length. We then extend the linear upper bounds to hyperbolic manifolds with a finite cover that admits such an immersion. Since the quantifications are given in terms of geodesic length, we define the geodesic residual finiteness growth and show that this growth is equivalent to the usual residual finiteness growth defined in terms of word length. This equivalence implies that our results recover the quantification of residual finiteness from \cite{BHP} for hyperbolic manifolds that virtually immerse into a compact reflection orbifold.\end{abstract}

\maketitle

\section{Introduction}\label{sec:intro}

It is well-known that separability properties on groups have deep connections with basic problems in group theory. In the 1940's, Mal'cev demonstrated that separability properties like residual finiteness and conjugacy separability produce solutions to the word and conjugacy problems, respectively, for finitely presented groups {\cite{Malcev}}. In more recent years, separability properties have also played a fundamental role in low dimensional topology. In \cite{Scott}, P. Scott gave an important topological reformulation of subgroup separability when the groups in consideration are the fundamental groups of manifolds; separability allows one to promote an immersed compact set to an embedded one in a finite cover. This topological reformulation of subgroup separability, usually called Scott's criterion, played a crucial role in the recent resolutions of Walhausen's Virtually Haken Conjecture and Thurston's Virtually Fibered Conjecture (see \cite{AgolGrovesManning}, \cite{Wise:QCH}, \cite{Haglund}). 

The simplest separability property, residual finiteness, allows us to separate nontrivial group elements from the identity using finite index subgroups. More precisely, a group $G$ is \emph{residually finite} if for every non-identity element $g \in G$, there exists a finite index subgroup $G'$ of $G$ such that $g \notin G'$. \emph{Quantifying residual finiteness}, a concept first introduced by K. Bou-Rabee in \cite{Bou1}, refers to bounding the indexes of the finite index subgroups $G'$ in terms of algebraic data about $G$. In studying separability properties of the fundamental groups of manifolds, the bounds can also be given in terms of geometric data about the manifolds as in \cite{PP12}. A quantification of residual finiteness informs us on the minimal possible index of a subgroup $G'$ of $G$ that separates $g$ from the identity, and it has been studied for various classes of groups including free groups, surface groups, and virtually special groups (see for instance \cite{Bou1}, \cite{BHP}, \cite{Bou2}, \cite{BM2}, \cite{Buskin}, \cite{Kass}, \cite{KozmaThom}, \cite{PP12}, \cite{Rivin}).

The reasons for quantifying residual finiteness and studying residual finiteness growths are varied. First, residual growths can help distinguish classes of groups. For example, in \cite{BM14} Bou-Rabee and McReynolds give a characterization of when non-elementary hyperbolic groups are linear in terms of residual growths. Additionally, a quantification of residual finiteness can serve as a foundation on which to build an approach to the quantification of stronger separability properties. In \cite{PP12}, the author presented a quantification of the residual finiteness of hyperbolic surface groups in terms of geodesic length. The result was then used to make effective a theorem of P. Scott \cite{Scott} that these groups are also subgroup separable. The author used a key insight of Scott that all hyperbolic surface groups arise as subgroups of a particular right-angled Coxeter group, generated by reflections in the sides of a regular, right-angled pentagon in ${\bf H}^2$. Quantification proofs also usually proceed algorithmically and can provide insight into how to construct the finite index subgroups/covers associated to residual finiteness and subgroup separability.

To keep with the notation present in most of the literature on quantifying residual finiteness we introduce the \emph{residual finiteness growth}, originally defined in \cite{Bou1} as follows. For a group $G$ with a fixed finite generating set $\mathcal S$, let the \emph{divisibility function} $D_G : G \setminus \{1\} \longrightarrow {\bf N} \cup \{\infty\}$ be defined by \[ D_G(g) = \min \{ [ G : H] : g \notin H \text{ and } H \leq G \}.\]

\noindent When $G$ is residually finite, $D_G(g)$ of course takes values in ${\bf N}$. We note that hyperbolic manifold groups are finitely generated linear groups and are thus known to be residually finite by the work of Mal'cev \cite{Malcev}. Define the residual finiteness function $\DM_{G,\mathcal S}(n)$ to be the maximum value of $D_G$ on the set
\[ \left\{ g \in G - \{1\} :\| g  \|_{\mathcal S} \leq n \right \}, \]
\noindent
where $\| \cdot \|_\mathcal{S}$ is the word-length norm with respect to $\mathcal S$. The growth of $\DM_{G,\mathcal S}$ is called the \emph{residual finiteness growth}.

In this paper, we are concerned with the residual finiteness growths of the fundamental groups of particular hyperbolic manifolds, and we give all quantifications of residual finiteness in terms of geometric data about the manifolds. Accordingly, we define a new function $\DM_{M, \rho}$ for a hyperbolic manifold $(M, \rho)$ by letting $\DM_{M, \rho}(n)$ be the maximum value of $D_{\pi_1(M)}$ on the set 
\[
\left\{ \alpha \in \pi_1(M) - \{1\} : \ell_\rho(\alpha) \leq n \right \},
\]
\noindent where $\ell_\rho(\alpha)$ is the length of the unique geodesic representative of $\alpha$ with respect to the hyperbolic metric $\rho$ on $M$. The growth of $\DM_{M,\rho}$ is called the \emph{geodesic residual finiteness growth}. 

In order to state our results in the language of growth functions, we introduce the following standard notation: For functions $f, g: {\bf N} \rightarrow {\bf N}$ , we write $f \preceq g$ if there exists $C > 0$ such that $f(n) \leq C\cdot g (Cn)$. Further, we write $f \simeq g$ if $f \preceq g$ and $g \preceq f$. 

When $M$ is a compact hyperbolic manifold, the residual finiteness growth and the geodesic residual finiteness growth are the same, i.e. $\DM_{{\pi_1(M)},\mathcal S} \simeq \DM_{M, \rho}$, due to an application of the Svarc-Milnor Lemma (see \cite[P. 140]{BridsonHaefliger}). However, we believe that the relationship between the two growths for non-compact, finite volume hyperbolic manifolds might be more complicated. We give a proof of the equality of the growths in the compact case (see Lemma \ref{lem:equivalent}) and a brief justification of the reasoning for the non-compact case in Section \ref{sec:comparison}.

This paper studies the residual finiteness of the fundamental groups of hyperbolic manifolds that admit a totally geodesic immersion to a compact, right-angled Coxeter orbifold of dimension 3 or 4. Begin by letting $P$ be any compact polyhedron in ${\bf H}^3$ (resp. ${\bf H}^4$), all of whose dihedral angles are ${\pi}/{2}$, which we will refer to as a compact \emph{all right polyhedron}. These polyhedra serve as the analog of Scott's right-angled pentagon in \cite{Scott}. (Note that compact all right polyhedra only exist in dimension up to 4 by the work of Vinberg \cite{Vinberg}). We denote by $\Gamma_P$ the right-angled Coxeter group of isometries of ${\bf H}^3$ (resp. ${\bf H}^4$) generated by reflections in the codimension--1 faces of $P$. The quotient ${\bf H}^3/\Gamma_P$ (resp. ${\bf H}^4/\Gamma_P$) is the compact, right-angled Coxeter orbifold, $\mathcal{O}_P$, defined by $P$. The orbifold $\mathcal{O}_P$ is also often called a compact reflection orbifold in the literature.

Drawing from the work of Agol, Long, and Reid in \cite{ALR}, and making use of an observation of Agol in \cite{Agol}, we obtain the following theorems, which constitute the main results of the paper and which generalize \cite[Theorem 6.1]{PP12}:

\begin{thma}\label{thm:quantrf}
Let $(M, \rho)$ be a hyperbolic manifold that admits a totally geodesic immersion to a compact, right-angled Coxeter orbifold, $\mathcal{O}_P$, of dimension 3. Then for any $\alpha \in \pi_1(M)-\{1\}$, there exists a subgroup $H'$ of $\pi_1(M)$ such that $\alpha \notin H'$, and the index of $H'$ is bounded above by \[ \dfrac{2\pi}{V_P}\sinh^2(\ln(\sqrt{3} + \sqrt{2}) + d_P) \,\ell_\rho(\alpha),\] where $\ell_\rho(\alpha)$ is the length of the unique geodesic representative of $\alpha$, and where $d_P$ and $V_P$ are the diameter and volume of $P$, respectively. 

\end{thma}

In Section \ref{sec:4DRF}, we establish a 4--dimensional analog of Theorem \ref{thm:rfmain}: 

\begin{thmb}\label{thm:quantrf4d}
Let $(M, \rho)$ be a hyperbolic manifold that admits a totally geodesic immersion to a compact, right-angled Coxeter orbifold, $\mathcal{O}_P$, of dimension 4.Then for any $\alpha \in \pi_1(M)-\{1\}$, there exists a subgroup $H'$ of $\pi_1(M)$ such that $\alpha \notin H'$, and the index of $H'$ is bounded above by \[ \dfrac{8\pi }{3V_P}\sinh^3(\ln(2 + \sqrt{3}) + d_P) \,\ell_\rho(\alpha),\] where $\ell_\rho(\alpha)$ is the length of the unique geodesic representative of $\alpha$, and where $d_P$ and $V_P$ are the diameter and volume of $P$, respectively. 

\end{thmb}

In the notation defined above, we then have the following corollary:

\begin{corollaryc}\label{cor:rfgrowth}
Let $(M, \rho)$ be a hyperbolic manifold admitting a totally geodesic immersion to a compact, right-angled Coxeter orbifold of dimension 3 or 4. Then, the geodesic residual finiteness growth is at most linear. That is to say, $\DM_{M, \rho} \preceq n$.
\end{corollaryc}

\begin{rem} We note that in the above theorems and corollary, the hyperbolic manifold $M$ need not be compact. In dimension 2, the compactness criteria of \cite[Theorem 5.4, Theorem 6.1, and Theorem 7.1]{PP12} is also not necessary. Indeed, every hyperbolic infinite area surface of finite type can be tiled by regular, right-angled pentagons and all proofs follow as in the compact case.

\end{rem}

\begin{rem} We also note that the dimension of the hyperbolic manifold $M$ need not match the dimension of the orbifold to which it admits a totally geodesic immersion. This is important since \cite[Lemma 4.6]{ALR} shows that there exist infinitely many compact, arithmetic hyperbolic 3--manifolds that admit a totally geodesic immersion to a particular compact reflection orbifold of dimension 4 (coming from the 120-cell in ${\bf H}^4$), but do not admit such an immersion to any reflection orbifold of dimension 3. 

\end{rem}

The next lemma allows us to extend Corollary \ref{cor:linear} to hyperbolic manifold groups that \emph{virtually} admit the desired type of immersion. A group $G$ (resp. topological space $X$) \emph{virtually} has property ``$\mathcal X$'' if there exists a  finite index subgroup of $G$ (resp. finite sheeted cover of $X$) with property ``$\mathcal X$".

\begin{lemmad}\label{lem:subgroups}
Let $(M, \rho)$ be a hyperbolic manifold and let $K \leq \pi_1(M)$ be a finite index subgroup with $[\pi_1(M) : K] = C$. Let $(M', \rho')$ be the cover of $M$ of degree $C$ corresponding to the subgroup $K$. Then the geodesic residual finiteness function for $(M,\rho)$ is bounded by that of $(M', \rho')$. That is to say, $\DM_{M, \rho} \leq C\cdot \DM_{M', \rho'}$ and hence $\DM_{M, \rho} \preceq \DM_{M', \rho'}$.
\end{lemmad}

\noindent Combining Theorems \ref{thm:rfmain} and \ref{thm:rfmain2} with Lemma \ref{lem:boundextends} gives:

\begin{corollarye}\label{cor:virtual}
If $(M, \rho)$ is a hyperbolic manifold that virtually admits a totally geodesic immersion to a compact, right-angled Coxeter orbifold of dimension 3 or 4, then $\DM_{M, \rho} \preceq n$.
\end{corollarye}

We note that in \cite{BHP}, K. Bou-Rabee, M.F. Hagen, and the author give a quantification of residual finiteness for right-angled Artin groups (raAgs) in terms of word length. In particular, we prove that the residual finiteness growth of raAgs is at most linear. As suggested by Lemma\;\ref{lem:boundextends}, residual finiteness quantifications are essentially preserved under passing to subgroups and finite index extensions. Thus, the quantification for raAgs results in a quantification of residual finiteness for all groups that virtually embed in raAgs, which are called \emph{virtually special groups}. This class of groups includes Coxeter groups \cite{HaglundWiseCoxeter}, and in particular contains the class of manifold groups satisfying the hypothesis of Corollary \ref{cor:virtually}. By an application of Lemma \ref{lem:equivalent} below, the results of \cite{BHP}, therefore, imply Corollary \ref{cor:virtually}. However, the proofs of Theorems \ref{thm:rfmain}  and \ref{thm:rfmain2} do not rely on the canonical completion and retraction methods of \cite{haglundwise} for cube complexes. Instead, the proofs use simple calculations in hyperbolic space inspired by work of Agol--Long--Reid \cite{ALR} and Agol \cite{Agol}. Most importantly, the calculations lead to the \emph{explicit} bounds of Theorem \ref{thm:rfmain} and \ref{thm:rfmain2}, and obtaining such bounds using \cite{BHP} should require a significant amount of work.

\vspace{.25 in}

\noindent {\bf Acknowledgements. }{Much of the work in this paper originally appeared in the author's thesis, written under the supervision of Feng Luo, whom the author thanks for his help, support, and encouragement. The author would also like to sincerely thank Ian Agol for sharing his work and ideas, as well as for insightful conversations. Many thanks are also due to Tian Yang for his suggestions regarding the proofs of Lemmas \ref{lem:nottoofar} and \ref{lem:volume4}, to Alan Reid and Nicholas Miller for helpful conversations, to Ben McReynolds for his encouragement to write this paper, and to David Duncan for comments on an early draft. The author would especially like to thank an anonymous referee whose extensive and thorough suggestions have greatly increased the quality of this paper.}

\section{Preliminaries}
\subsection{General notation:}\label{sec:notation} Throughout the paper we switch freely between the Poincar\'{e} ball (${\bf D}^n$) and half-space (${\bf H}^n$) models of hyperbolic space. In the 3--dimensional Poincar\'{e} ball model ${\bf D}^3$, we define the three hyperplanes $L_x = \{(x, y, u) \in {\bf D}^3: x = 0 \}$, $L_y= \{(x, y, u) \in {\bf D}^3: y = 0 \}$, and $L_u = \{(x, y, u) \in {\bf D}^3: u = 0 \}$. In the 3--dimensional half-space model ${\bf H}^3$ we define the hyperplanes $L'_x = \{(x, y, u) \in {\bf H}^3: x = 0 \}$ and $L'_y= \{(x, y, u) \in {\bf H}^3: y = 0 \}$. Similarly, in ${\bf D}^4$ we define the four hyperplanes $L_{x_i}$ for $i = 1, 2, 3, 4$ as the hyperplanes obtained by restricting $x_i$ to be $0$. The hyperplanes $L'_{x_k}$ for $k = 1, 2, 3$ in ${\bf H}^4$ are defined by restricting the $x_k$ coordinate in ${\bf H}^4$ to be $0$.

\subsection{Topological view of residual finiteness:} When G is the fundamental group of a hyperbolic manifold, we have the following topological formulation of residual finiteness, which is used to prove Theorems \ref{thm:rfmain} and \ref{thm:rfmain2}: a hyperbolic manifold group $\pi_1(M)$ is residually finite if for every $\alpha \in \pi_1(M) - \{1\}$ there exists a finite index cover $\widetilde M$ of $M$ where the unique geodesic representative of $\alpha$ does not lift (i.e. no component of the preimage of $\alpha$ in $\widetilde{M}$ projects injectively to  $\alpha$).

\subsection{Polyhedral convexification:}

If $P$ is an all right polyhedron in ${\bf H}^n$, then by the Poincar\'{e} Polyhedron Theorem, the images of $P$ under the action of $\Gamma_P$ will tessellate ${\bf H}^n$.

\begin{definition} 
The \emph{$P$--convexification} of a connected set $\mathcal K$ in ${\bf H}^n$, denoted by $C_P(\mathcal K)$, is the smallest, convex union of polyhedra in the tessellation of ${\bf H}^n$ determined by $P$ that contains $\mathcal K$.
\end{definition}

\noindent Equivalently, we can define the convexification as the intersection of all half spaces bounded by the geodesic hyperplanes in our tessellation of ${\bf H}^n$ determined by $P$ that contain $\mathcal K$. \\

\section{The Proof of Theorem \ref{thm:rfmain}}\label{sec:3DRF}

Let $(M, \rho)$ be a hyperbolic manifold that admits a totally geodesic immersion, $f$, to a compact, right-angled Coxeter orbifold, $\mathcal{O}_P$, of dimension 3. Then the induced map $f_*: \pi_1(M) \longrightarrow \pi_1(\mathcal{O}_P) = \Gamma_P$ is injective, and $\pi_1(M)$ can be identified with a subgroup of $\pi_1(\mathcal{O}_P) = \Gamma_P$. When convenient, we blur the distinction between $\pi_1(M)$ and its image $f_*(\pi_1(M))$ in $\Gamma_P$ and simply write $\pi_1(M) < \Gamma_P$. Additionally, the length of the geodesic representative of $\alpha \in \pi_1(M) - \{1\}$ will be equal to the translation length of $f_*(\alpha) \in \Gamma_P$. 

For $\alpha' \in \pi_1(M) - \{1\}$ with $f_*(\alpha') = \alpha \in \Gamma_P$, the cyclic subgroup of $\Gamma_P$ generated by $\alpha$ will be denoted by $\Phi = \langle \alpha \rangle$. We let $\overline \alpha$ be the unique simple closed geodesic of $X = {{\bf H}^3}/\Phi$. 

The preimage of $\overline{\alpha}$ under the covering map ${\bf H}^3 \rightarrow X$ is the geodesic axis, $\widetilde Y$, in ${\bf H}^3$ for the element $\alpha$. We note that $\widetilde Y$ is invariant under the action of $\langle \alpha \rangle$ on ${\bf H}^3$, and therefore, so is $C_P(\widetilde Y)$, the $P$-convexification of $\widetilde{Y}$. The image of $C_P(\widetilde Y)$ in $X$ is denoted by $\mathcal C$. Thus, $\mathcal C \subset X$ is the smallest closed, connected, convex union of polyhedra in the tessellation of $X$ (determined by $P$) containing $\overline \alpha$. We refer to $\mathcal C$ as the \emph{convexification of $\overline \alpha$ in $X$}. The following lemma, along with Lemma \ref{lem:Rcalculation} in the next section, were first proposed by Agol via private communication \cite{Agol}.

\vspace{.1 in}

\begin{lemma} \label{lem:nottoofar}
Let $\mathcal C$ be the union of polyhedra forming the convexification of \,$\overline \alpha$ in $X = {\bf H}^3/\Phi$ via the procedure mentioned above. Then any polyhedron $P_i \in \mathcal C$ must intersect $N = N_R(\overline \alpha)$ where $R = \ln(\sqrt{3} + \sqrt{2})$.
\end{lemma}

\begin{proof}
We define $\widetilde N$ to be the preimage of $N$ in ${\bf D}^3$. Thus, $\widetilde N$ forms an $R$--neighborhood around the geodesic axis $\widetilde Y$ in ${\bf D}^3$. Suppose $P$ is a polyhedron in our tessellation of ${\bf D}^3$ which does not intersect $\widetilde N$. We aim to show that a hyperplane in ${\bf D}^3$ containing one of the codimension--1 faces of $P$ must separate $P$ from $\widetilde Y$, demonstrating that $P \notin \widetilde{\mathcal{C}} = C_P(\widetilde Y)$ and proving the lemma. 

We let $e$ be the minimal-dimensional face of $P$ containing its closest point to $\widetilde Y$ and set $k = $codim$(e)$. Note that $k$ is also the number of codimension--1 faces of $P$ that intersect $e$. Take a shortest geodesic $\overline{py}$ from $e$ to $\widetilde Y$ that intersects $\widetilde Y$ at a point $y$ and $e$ at a point $p$. Then $\overline{py}$ is an orthogeodesic with $\ell(\overline{py}) > R$ since $P$ does not intersect the $R$--neighborhood of $\widetilde Y$. We let $j$ be a hyperplane through $y$ that is perpendicular to $\overline{py}$, which separates $e$ from $\widetilde Y$. We note that $j$ necessarily contains $\widetilde Y$.

\vspace{.2 in}

{\bf Case $k = 3$:} We first consider the case where $e$ is a vertex of $P$ and $k = 3$. We begin by sending $e = p$ to the origin of ${\bf D}^3$ via isometries. Since $P$ is an all right polyhedron, the three codimension--1 faces of $P$ that intersect $e$ necessarily lie in three hyperplanes $L_1$, $L_2$ and $L_3$ that are the isometric image of $L_x$, $L_y$ and $L_u$, which were defined in Section~\ref{sec:notation}. Letting $\partial L_i = \partial {\bf D}^3 \cap L_i$, we see that the connected components of $\partial{{\bf D}^3} - \{\partial L_1\cup \partial L_2 \cup \partial L_3\}$ are right-angled spherical triangles in $\partial {\bf D}^3$ with edge lengths $\frac{\pi}{2}$. 

We note that we can assume that $\widetilde Y$ passes through the north pole of the hyperplane $j$ (this will be important for our calculation of $R$). We apply isometries to ${\bf D}^3$ until $y$, and therefore $\overline{py} = \overline{0y}$, lies on the line formed by $L_u \cap L_x$ and $\widetilde Y$ lies in $L_x$. The distance $d(P, \widetilde Y)$ and the radius, $r$, of $j \cap \partial {\bf D}^3 := \partial j$ are inversely proportional. Indeed, Figures \ref{fig:partialj} and \ref{fig:partialj2} show that as the distance between $P$ and $\widetilde Y$ (the length of $\overline{0y}$) grows, the radius of $\partial j$ becomes small.

\begin{figure}[h]
\centering
\subfloat[]{\label{fig:partialj}\includegraphics[trim = 1in .5in 0in .5in, clip=true, totalheight=0.2\textheight]{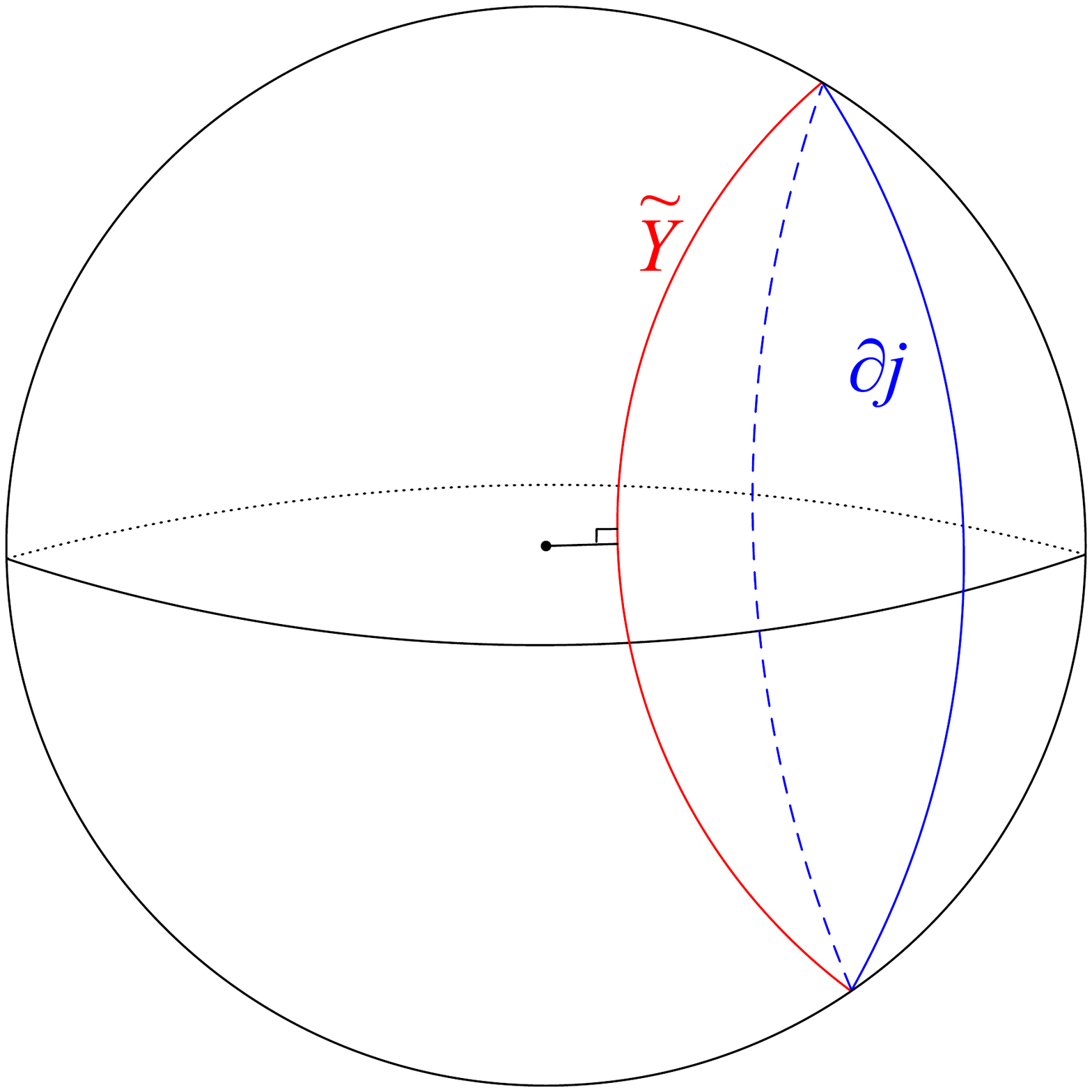}}
\subfloat[]{\label{fig:partialj2} \includegraphics[trim = 0in .5in 1in .5in, clip=true, totalheight=0.2\textheight]{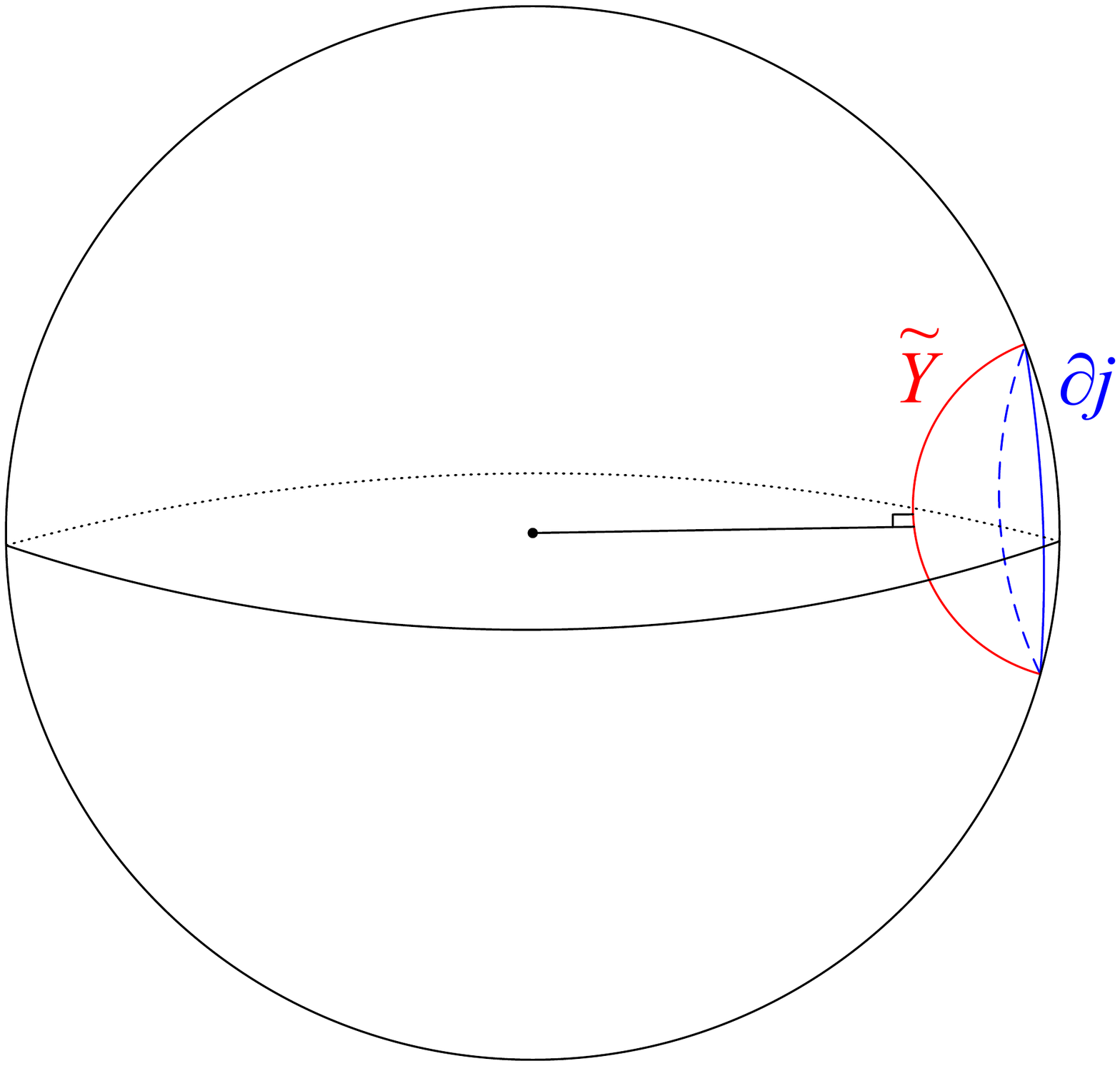}}
\caption{} 
\end{figure}

We are therefore interested in the threshold $R$ of the distance between $P$ and $\widetilde Y$ so that $\partial j$ can be inscribed in a right-angled spherical triangle formed by the boundaries of three pairwise orthogonal hyperplanes. Then if $d(P, \widetilde Y) > R$, at least one of the three circles $\partial L_i$ cannot intersect $\partial j$, and thus, one of the three hyperplanes $L_i$ must separate $P$ from $\widetilde Y$. 

We begin by calculating the radius of a circle inscribed in such a spherical triangle, whose edge lengths are $\frac{\pi}{2}$. This calculation follows from an application of classical formulas (see e.g. \cite[Section 89]{Todhunter}), but we also include it here. In Figure \ref{fig:inscribed}, $A$, $B$ and $C$ are the midpoints of the three edges in our spherical triangle, which are also the points of tangency for the inscribed circle. Since the triangle is formed by the intersection of three pairwise orthogonal hyperplanes in ${\bf D}^3$ with $\partial {\bf D}^3$, the unit sphere in ${\bf R}^3$, the length of the edge $\overline{DE}$ is equal to $\frac{\pi}{2}$, and the length of $\overline{DB}$ is $\frac{\pi}{4}$. 

\begin{figure}[h]
\centering
\subfloat[]{\label{fig:inscribed}\includegraphics[trim = 1.5in .5in 1in .5in, clip=true, totalheight=0.25\textheight]{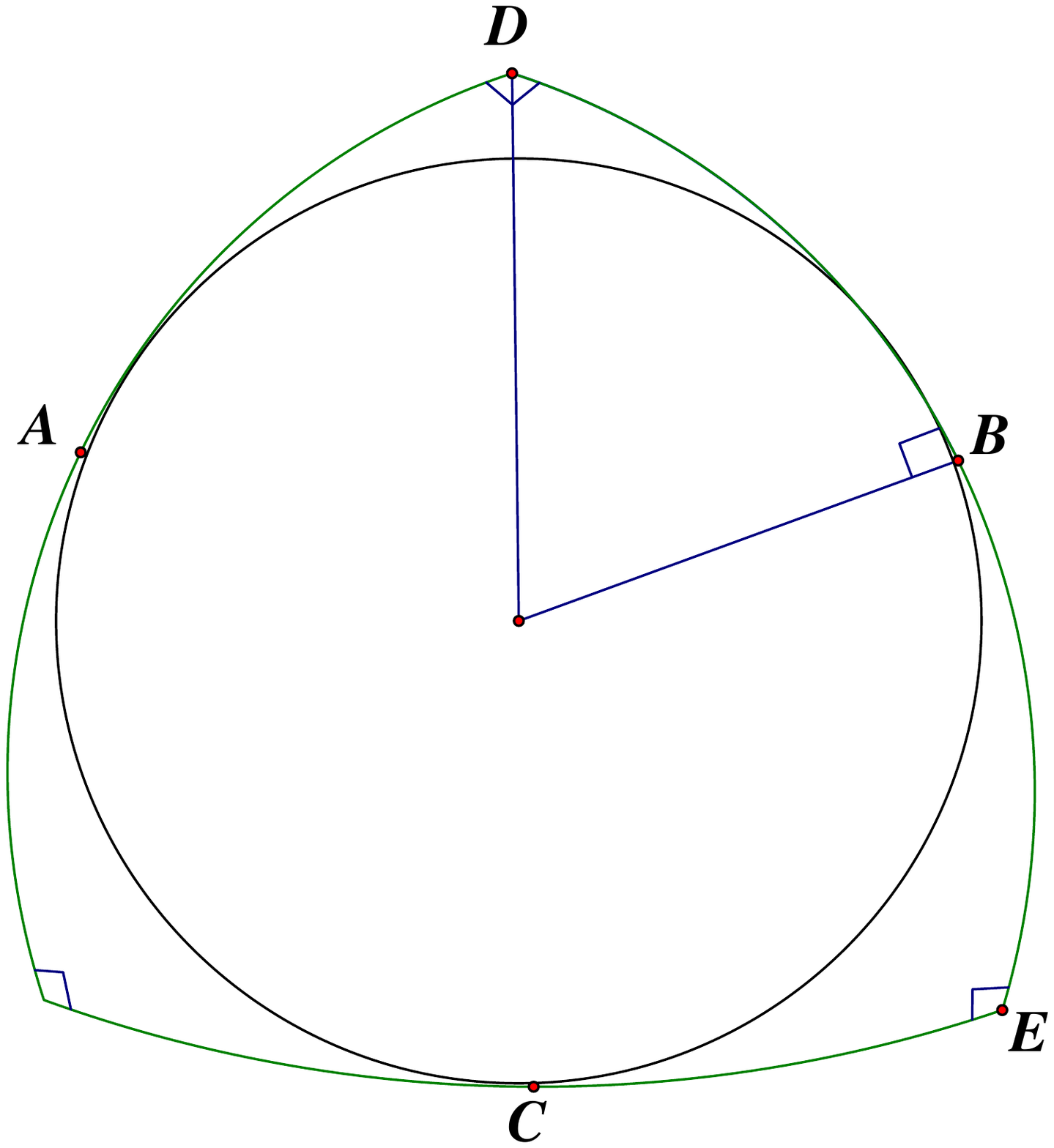}}
\subfloat[]{\label{fig:inscribed2} \includegraphics[trim = 0in .5in 1in .5in, clip=true, totalheight=0.23\textheight]{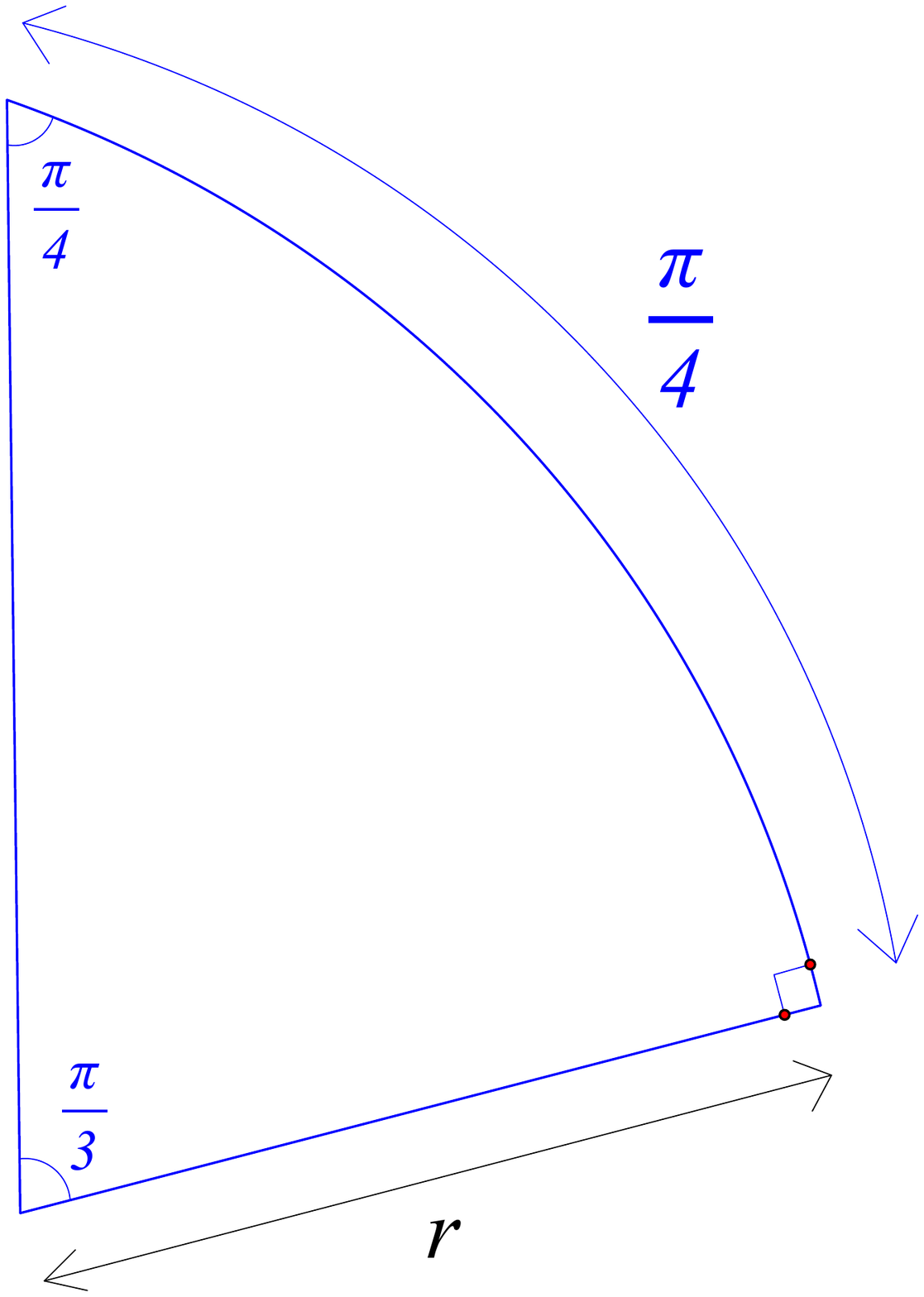}}
\caption{} 
\end{figure}

To calculate the radius $r$ of the inscribed circle we apply the following Spherical Law of Cosines. Let $T$ be a spherical triangle with angles $\alpha$, $\beta$ and $\gamma$, and with edges of lengths $a$, $b$ and $c$ opposite the angles $\alpha$, $\beta$ and $\gamma$, respectively. Then,

\[\cos \alpha = - \cos \beta \cos \gamma + \sin \beta \sin \gamma \cos a.\]

\noindent For the triangle in Figure \ref{fig:inscribed2}, we have $\cos \frac{\pi}{4} = \sin \frac{\pi}{2} \sin \frac{\pi}{3} \cos r$, and therefore, $r = \cos^{-1} \left(\frac{\sqrt{2}}{\sqrt{3}}\right)$.

Now we calculate the distance $R =$ length of $\overline{0y} = d(P, \widetilde Y)$ so that the radius $r$ of $\partial j$ is $\cos^{-1} \left(\frac{\sqrt{2}}{\sqrt{3}}\right)$. Consider the cross sectional view formed by the intersection of the hyperplane $L_x$ with our setup in Figures \ref{fig:partialj}, \ref{fig:partialj2} above. This view is represented in Figure \ref{fig:crosssec}. 

\begin{figure}[h]
\centering
{\includegraphics[trim = 1in .5in 1in .5in, clip=true, totalheight=0.25\textheight]{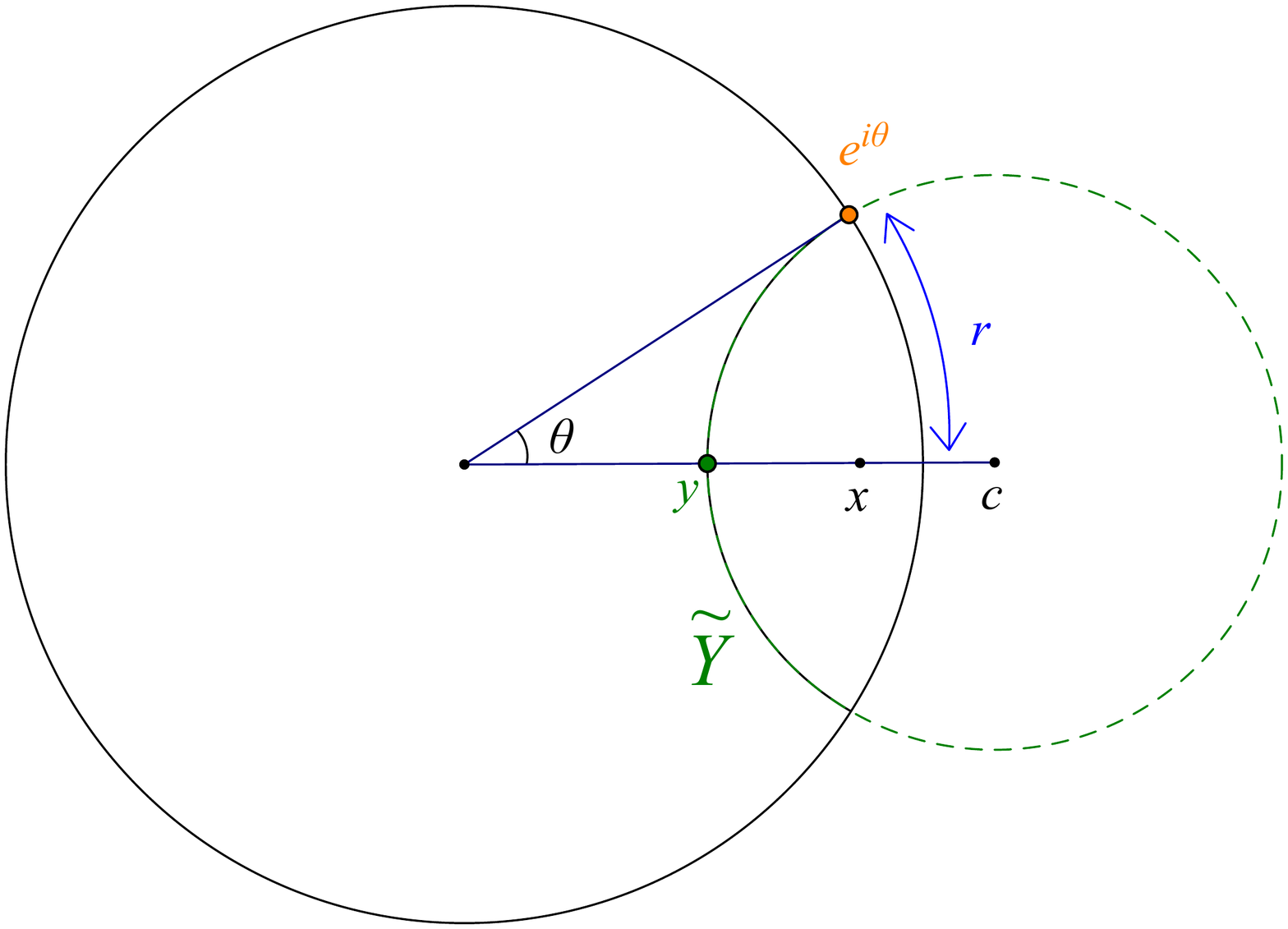}}
\caption{} 
\label{fig:crosssec}
\end{figure}

Since $\partial {\bf D}^3 \cap \partial L_x$ is a unit circle we know that $\theta = r = \cos^{-1} \left(\frac{\sqrt{2}}{\sqrt{3}}\right)$. Note that the points $0, e^{i\theta}, \text{and } c$ form a right triangle and we denote the length of the line segment from $e^{i\theta}$ to $c$ by $q$. Therefore, $c = \sec(\theta)= \frac{\sqrt{3}}{\sqrt{2}}$ and $q= \tan(\theta) = \frac{1}{\sqrt{2}}$. Lastly, $y = c - q = \frac{\sqrt{3} - 1}{\sqrt{2}}$ and $R = d(0, \widetilde Y) = \ln\left(\frac{1 + y}{1 - y}\right)$, which by a simple calculation gives us $R = \ln(\sqrt{2} + \sqrt{3})$.

\vspace{.2 in}

{\bf Case $k = 2$:} The case where $e$ is an edge of $P$ can be handled in a similar way. In this case, we show that $R = \ln(\sqrt{2} + 1)$.

We assume the same setup as in the previous case where the point $p$ on the edge $e$ that is closest to $\widetilde Y$ is at the origin of ${\bf D}^3$ and $y$ lying on $L_u \cap L_x$. The extensions of the two codimension--1 faces of $P$ that intersect $e$ form a pair of orthogonal hyperplanes, $L_1$ and $L_2$, in ${\bf D}^3$. Their boundaries, $\partial L_1$ and $\partial L_2$, form a spherical bi-disk with angles $\frac{\pi}{2}$. We are looking for the threshold $R$ such that $\partial j$, and thus $\widetilde Y$, is tangent to such a spherical bi-disk at the endpoints of $\widetilde Y$. Then, if $d(0, \widetilde Y) > R$, one of the $\partial L_i$ cannot intersect $\partial j$, and one of the hyperplanes $L_i$ must therefore separate $P$ from $\widetilde Y$. 

A cross sectional view of this situation is shown in Figure \ref{fig:case2} below. Given the triangle in the figure, we know that $c = \sqrt{2}$ so that $y = \sqrt{2} - 1$. Thus, $R = \ln \left(\frac{1 + \sqrt{2} - 1}{1 - \sqrt{2} +1}\right) = \ln(\sqrt{2} + 1)$.

\begin{figure}[h]
\centering
\includegraphics[trim = 1in 1.5in 1in 1.5in, clip=true, totalheight=0.18\textheight]{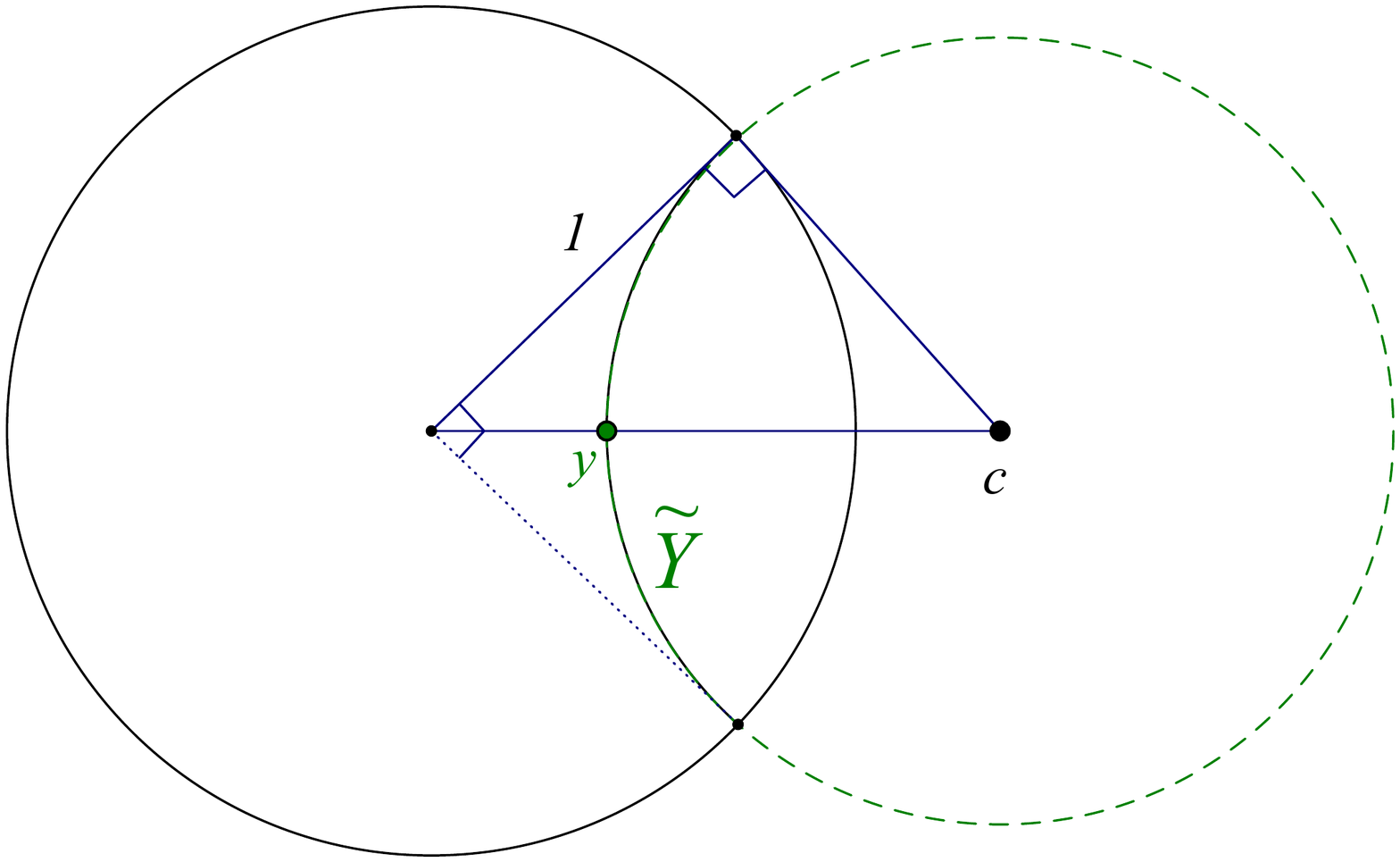}
\caption{}
\label{fig:case2}
\end{figure}

\vspace{.2 in}

{\bf Case $k = 1$:} Lastly, we consider the case where $e$ is a codimension--1 face of $P$, that is the case where $k = 1$. Then $\overline{py}$ is an orthogeodesic between the hyperplane containing $e$, which we also call $e$ for notational simplicity, and the hyperplane $j$. Taking any $R>0$  is sufficient in this case since $e$ itself is a codimension--1 face of $P$ whose hyperplane extension separates $P$ from $\widetilde Y$ (see Figure \ref{figure:face} below).

\begin{figure}[h]
\centering
\includegraphics[trim = .75in 2in .75in 1.7in, clip=true, totalheight=0.18\textheight]{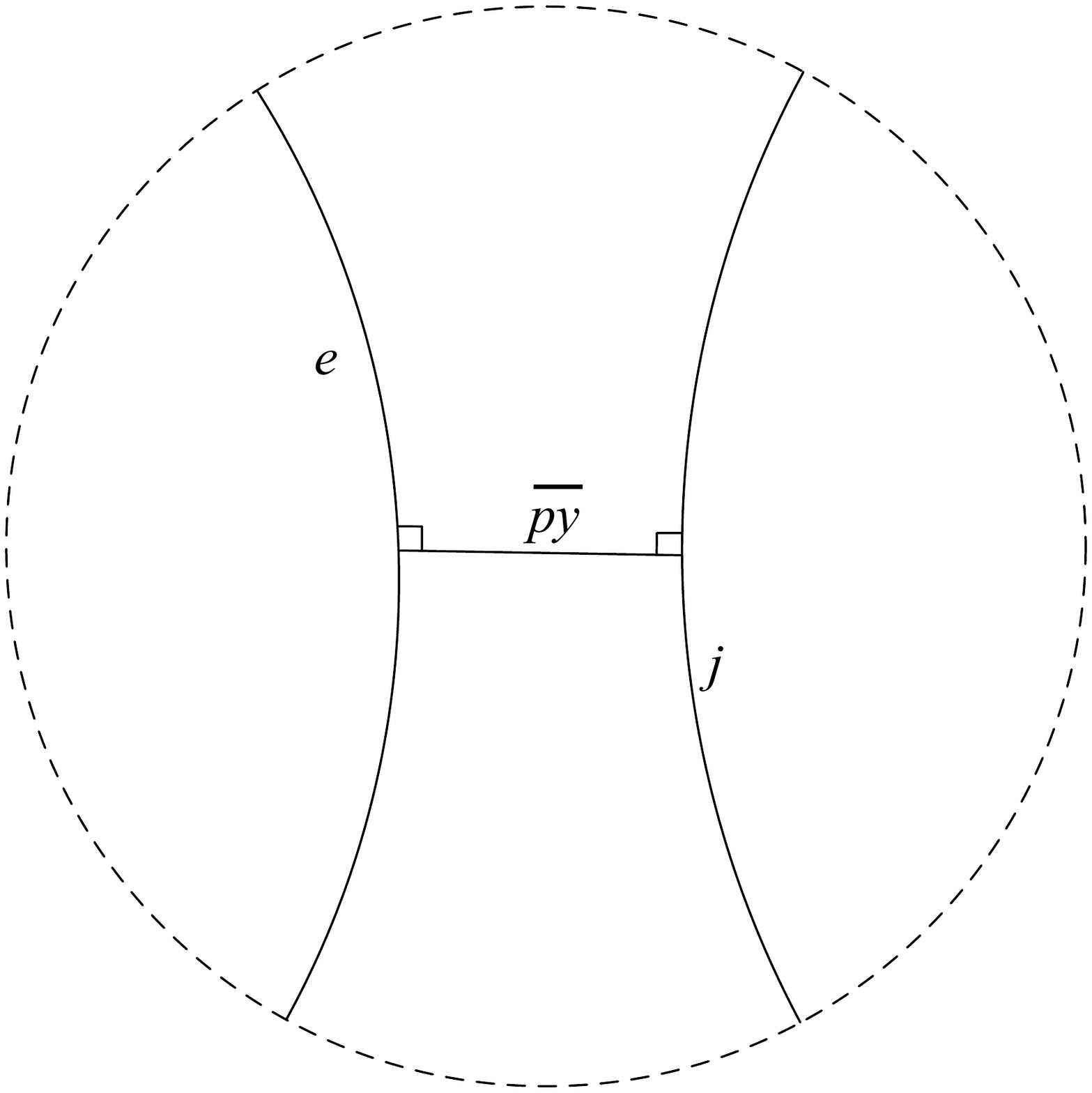}
\caption{}
\label{figure:face}
\end{figure}

We take $R$ to be the largest of the values in the three cases, i.e. $R = \ln(\sqrt{2} + \sqrt{3})$. If $P$ does not intersect $\widetilde N = N_R(\widetilde Y)$, then there is a codimension--1 face of $P$ whose hyperplane extension separates $P$ from $\widetilde Y$ so that $P \notin \widetilde{\mathcal C}$. Therefore, any polyhedron $P_i$ in $X = {{\bf H}^3}/\langle \alpha \rangle$ that is in the convexification $\mathcal C$ of $\overline \alpha$ must intersect the $\ln(\sqrt{2} + \sqrt{3})$--neighborhood $N$ of $\overline \alpha$.

\end{proof}

Given the above lemma, we have that $\mathcal{C} \subset N_{R+d_P}(\overline \alpha)$, where $d_P$ is the diameter of $P$. The following lemma allows us to calculate the volume of $N_{R+d_P}(\overline \alpha)$. Again, this calculation is standard (e.g. it is stated without proof in \cite{Meyerhoff}), but it is included here to set the stage for the proof of Theorem \ref{thm:rfmain}.  
\vspace{.1 in}

\begin{lemma}\label{lem:volume3}
We let $\Omega$ be the solid tubular $b$--neighborhood of the geodesic segment in ${{\bf H}^3}$ between the points $(0, 0, R_0)$ and $(0, 0, r_0)$ as shown in Figure \ref{fig:tubular3}. Then \emph{Vol}$(\Omega) = \pi \sinh^2(b) \, \ell$, where $\ell = \ln(R_0/r_0)$ is the length of the geodesic between the points $(0, 0, R_0)$ and $(0, 0, r_0)$ in ${{\bf H}^3}$.

\begin{figure}[h]
\centering
 \includegraphics[trim = 2in 1in 2in .75in, clip=true, totalheight=0.22\textheight]{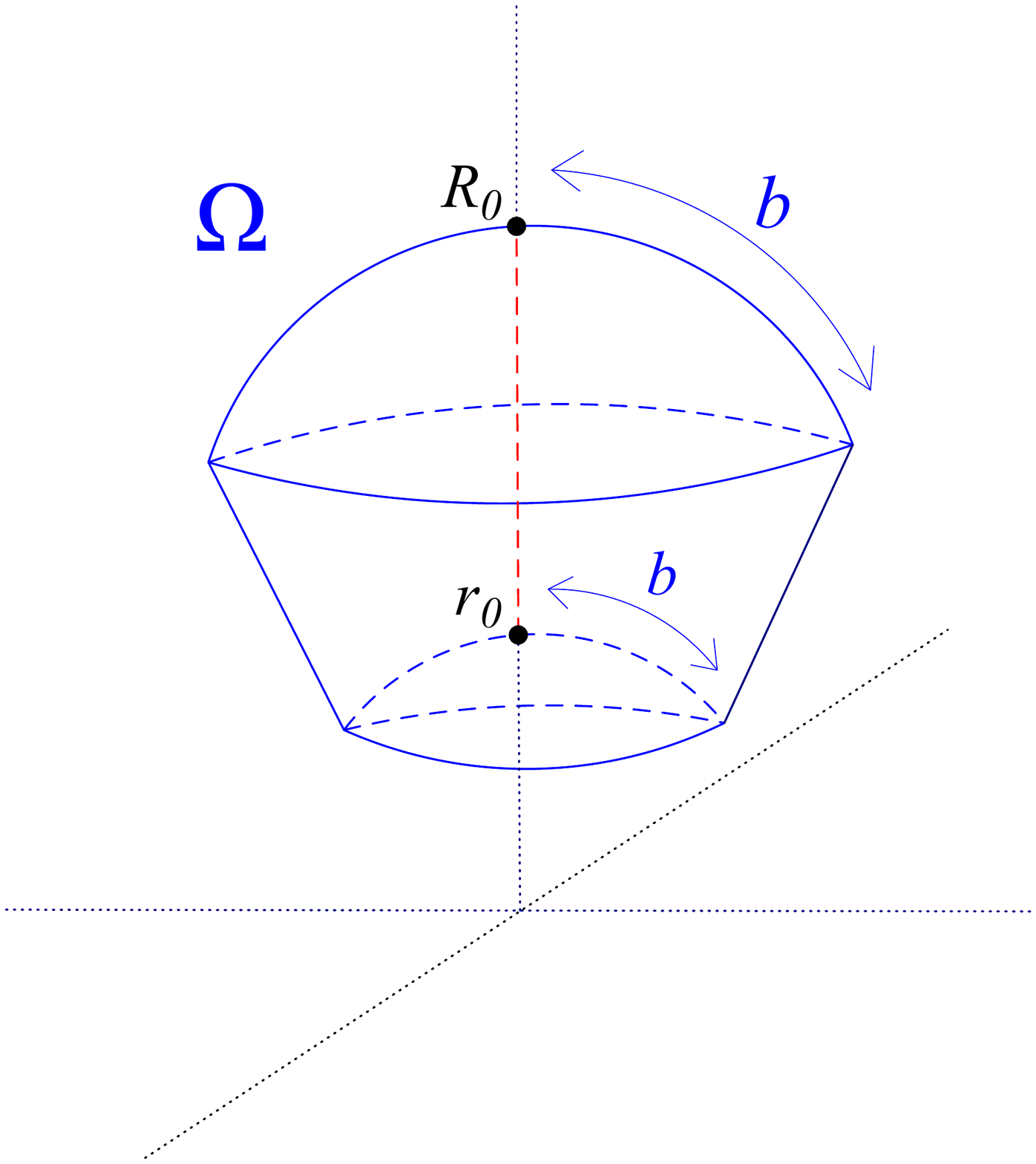}
\caption{}
\label{fig:tubular3}
\end{figure}

\end{lemma}

\begin{proof}
For this volume calculation we find it convenient to use spherical coordinates. The volume form on ${{\bf H}^3}$, $\frac{dx \wedge dy \wedge du}{u^3}$, becomes $\frac{1}{r} \tan \phi \sec^2 \phi \, dr \wedge d\phi \wedge d\theta$. Let $\pi(b)$ be the angle in $L'_y$ from the positive $x$-axis to the Euclidean ray consisting of points at hyperbolic distance $b$ from the $u$-axis, so that the range of values for $\phi$ in $\Omega$ is then $[\, 0, \, \pi/2 - \pi(b)\, ]$. Therefore,

\vspace{.1 in}

\begin{center} $\text{Vol}(\Omega) = \displaystyle\iiint \limits_{\Omega} \frac{1}{r} \, \, \tan \phi \, \, \sec^2 \phi \, \, dr \wedge d\phi \wedge d\theta \; \;$

\vspace{.1 in}

 $=\displaystyle\int_{0}^{2\pi} \displaystyle\int_{0}^{\frac{\pi}{2} - \pi(b)}\displaystyle\int_{r_0}^{R_0} \frac{1}{r} \tan \phi \, \,  \sec^2 \phi \, \, dr\, d\phi \, d\theta \; \; $
 
\vspace{.1 in}
 
$=\ln\left(\frac{R_0}{r_0}\right)  \displaystyle\int_{0}^{2\pi} \displaystyle\int_{0}^{\frac{\pi}{2} - \pi(b)}\tan \phi \,\,  \sec^2 \phi \, \, \, d\phi \, d\theta \; \; = \; \;  \ell \left[ \frac{\tan^2\phi}{2} \bigg|_{0}^{\frac{\pi}{2} - \pi(b)}\right]  \displaystyle\int_{0}^{2\pi} \, d\theta  \; \; $

\vspace{.1 in}

$=\frac{\pi}{\tan^2 \pi(b)} \, \ell \; \; = \; \; \pi \sinh^2(b) \, \ell$, 

\end{center}

\vspace{.1 in}

\noindent with the last equality coming from the angle of parallelism laws in hyperbolic space (see \cite[Section 7.9]{Beardon}).

\end{proof}

We take a subset of the preimage of $N_{R+d_P}(\overline \alpha)$ in ${{\bf H}^3}$ that is isometric to $N_{R+d_P}(\overline{\alpha})$ and which forms a region like $\Omega$ from the previous lemma, where $b = R + d_P$. Lemma \ref{lem:volume3} then implies that Vol$(\mathcal C) < $ Vol$(N_{R+d_P}(\overline \alpha)) = \pi \sinh^2(R + d_P) \, \ell_\rho(\alpha)$. Thus, we have the following theorem:

\begin{theorem}\label{thm:rfmain}
Let $(M, \rho)$ be a hyperbolic manifold that admits a totally geodesic immersion to a compact, right-angled Coxeter orbifold, $\mathcal{O}_P$, of dimension 3. Then for any $\alpha \in \pi_1(M)-\{1\}$, there exists a subgroup $H'$ of $\pi_1(M)$ such that $\alpha \notin H'$, and index of $H'$ is bounded above by \[ \dfrac{2\pi}{V_P}\sinh^2(\ln(\sqrt{3} + \sqrt{2}) + d_P) \,\ell_\rho(\alpha),\] where $\ell_\rho(\alpha)$ is the length of the unique geodesic representative of $\alpha$, and where $d_P$ and $V_P$ are the diameter and volume of $P$, respectively.

\end{theorem}

\begin{proof}
We know that Vol$(\mathcal C) < \pi \sinh^2(R + d_P) \, \ell_\rho(\alpha)$ so if $\mathcal C$ consists of $k$ polyhedra, \[k < \dfrac{\pi}{V_P}\sinh^2(\ln(\sqrt{3} + \sqrt{2}) + d_P) \,\ell_\rho(\alpha).\]

Let $\widetilde \alpha$ be one lift of $\overline \alpha$ to ${\bf H}^3$. By a lift of $\overline{\alpha}$, we mean a lift of the path $\overline{\alpha}: [0,1] \longrightarrow X$ to a path $\widetilde{\alpha}: [0,1] \longrightarrow {\bf H^3}$ starting at the basepoint $\widetilde{\alpha}(0) = x$. Using the right-angled tiling of $X$, we can lift $\mathcal C$ to ${\bf H}^3$ so that the result is a connected, convex union of $k$ polyhedra in $C_P(\widetilde Y)$ denoted by $\overline{\mathcal C}$, which contains the geodesic segment $\widetilde \alpha$. The convexity of the set is crucial since we will want to apply the Poincar\'{e} Polyhedron Theorem to prove the result above.

Set $\widetilde \alpha_1$ to be one of the two lifts of $\overline \alpha$ that share endpoints with $\widetilde \alpha$. If $\overline {\mathcal{C}_1}$ is the associated convex lift of $\mathcal C$ containing $\widetilde \alpha_1$, then $\mathcal C' = \overline{\mathcal C} \cup \overline{\mathcal{C}_1}$ is a convex union of $2k$ polyhedra in ${\bf H}^3$, such that one endpoint of $\widetilde \alpha$ is contained in the interior of $\mathcal C'$.

Denote by $H$ the group of isometries of ${\bf H}^3$ generated by reflections in the sides of $\mathcal C'$. Then $H < \Gamma_P$, and $\mathcal C'$ is a fundamental domain for the action of $H$ on ${\bf H}^3$ by the Poincar\'{e} Polyhedron Theorem. Since $\mathcal C'$ contains $2k$ polyhedra, $[\Gamma_P\!:\!H] = 2k$. Letting $p: {\bf H}^3 \longrightarrow {\bf H}^3/H$ be the covering map, we then have that the restriction of $p$ to the interior of ${\mathcal C'}$ is a homeomorphism onto its image in ${\bf H}^3/H$. Thus, $p(\widetilde \alpha)$ is not a loop in ${\bf H}^3/H$, and $\alpha \notin H$.

Now, if $H' = H \cap \pi_1(M)$, then $\alpha \notin H'$ and $[\pi_1(M)\!:\!H'] \leq [\Gamma_P\!:\!H] = 2k$. The result follows. Note that we have used both the fact that $\pi_1(M) < \Gamma_P$ and that the length of $\alpha$ is equal to to its translation length in $\Gamma_P$ in a crucial way. 

\end{proof}

\section{The Proof of Theorem \ref{thm:rfmain2}}\label{sec:4DRF}
In this section we obtain the analogous results of the previous section for hyperbolic manifolds that admit a totally geodesic immersion to a compact, right-angled Coxeter orbifold of dimension 4. We again denote the diameter and volume of $P$ by $d_P$ and $V_P$, respectively. As before, the images of $P$ under the action of $\Gamma_P$ tessellate ${\bf H}^4$. We have the following analog of Lemma \ref{lem:volume3}:

\begin{lemma}\label{lem:volume4}
Let $\Omega'$ be the analog in ${\bf H}^4$ of $\Omega$ from Lemma \ref{lem:volume3}. That is to say $\Omega'$ is a tubular $b$--neighborhood of the geodesic segment between the points $(0, 0, 0 , R_0)$ and $(0, 0, 0, r_0)$ lying on the $x_4$--axis in ${\bf H}^4$. Then \emph{Vol}$(\Omega') = \frac{4}{3}\pi \sinh^3(b) \, \ell$, where $\ell = \ln(R_0/r_0)$ is the length of the geodesic between the points $(0, 0, 0 , R_0)$ and $(0, 0 ,0, r_0)$ in ${\bf H}^4$. 

\end{lemma}

\begin{proof}
To calculate the volume of $\Omega'$ we again find it convenient to use generalized spherical coordinates. In ${\bf H}^4$ spherical coordinates are defined by:

\begin{align}  x_1 &= r \sin \phi_1 \sin \phi_2 \sin \phi_3  \nonumber \\
x_2 &= r \sin \phi_1 \sin \phi_2 \cos \phi_3 \nonumber \\
x_3 &=   r \sin \phi_1 \cos \phi_2 \nonumber \\
x_4 &=   r \cos \phi_1 \nonumber,
\end{align}

\noindent where $\phi_1 \in \left[0, \frac{\pi}{2}\right), \phi_2 \in [0, \pi], \phi_3 \in [0, 2\pi]$. 
Thus, the volume form on ${\bf H}^4$ is 

\[\frac{dx_1 \wedge dx_2 \wedge dx_3 \wedge dx_4}{x_4^4} = \left(\frac{1}{r} \tan^2 \phi_1 \sec^2 \phi_1 \sin \phi_2\right) \, dr \wedge d\phi_1 \wedge d\phi_2 \wedge d\phi_3 .\]
 As before, the range of values for $\phi_1$ in $\Omega'$ is $[\, 0, \, \pi/2 - \pi(b)\, ]$, where $\pi(b)$ is the angle in the plane $L'_{x_1} \cap L'_{x_2}$ from the positive $x_3$-axis to the Euclidean ray consisting of points at hyperbolic distance $b$ from the $x_4$-axis. Additionally, the range of values of $\phi_2$ and $\phi_3$ in $\Omega'$ are unrestricted so that $\phi_2$ takes values in $[0, \pi]$ and $\phi_3$ takes values in $[0, 2\pi]$. Thus, 

 \vspace{.1 in}

\begin{center} $\text{Vol}(\Omega') = \displaystyle\iiiint \limits_{\Omega'} \, \, \left(\frac{1}{r} \tan^2 \phi_1 \sec^2 \phi_1 \sin \phi_2\right) \, dr \wedge d\phi_1 \wedge d\phi_2 \wedge d\phi_3  \; \;$

\vspace{.1 in}

 $=\displaystyle\int_{0}^{2\pi} \displaystyle\int_{0}^{\pi} \displaystyle\int_{0}^{\frac{\pi}{2} - \pi(b)}\displaystyle\int_{r_0}^{R_0} \left(\frac{1}{r} \tan^2 \phi_1 \sec^2 \phi_1 \sin \phi_2\right) \, dr \, d\phi_1 \, d\phi_2 \, d\phi_3 \; \; $
 
\vspace{.1 in}
 
$=\ln\left(\frac{R_0}{r_0}\right) \displaystyle\int_{0}^{2\pi}  \displaystyle\int_{0}^{\pi} \displaystyle\int_{0}^{\frac{\pi}{2} - \pi(b)} \tan^2 \phi_1 \sec^2 \phi_1 \sin \phi_2 \, \, \, d\phi_1 \, d\phi_2 \, d\phi_3 \; \; $

\vspace{.1 in}

$= \; \;  \ell \left[ \frac{\tan^3\phi_1}{3} \bigg|_{0}^{\frac{\pi}{2} - \pi(b)}\right]  \displaystyle\int_{0}^{2\pi} \displaystyle\int_{0}^{\pi}  \sin \phi_2 \, \, d\phi_2 \, d\phi_3 \; \; $

\vspace{.1 in}

$=\frac{2\pi}{3\tan^3 \pi(b)} \, \ell \; \; \left[ -\cos \phi_2 \bigg|_{0}^{\pi} \right]   = \frac{4\pi}{3\tan^3 \pi(b)} \, \ell = \; \; \frac{4}{3} \pi \sinh^3(b) \, \ell$, 

\end{center}

\vspace{.1 in}

\noindent with the last equality coming from the angle of parallelism laws in hyperbolic space (see \cite[Section 7.9]{Beardon}). 

\end{proof}

Next we prove the analog of Lemma \ref{lem:nottoofar}. As in the previous section, $\Phi = \langle \alpha \rangle$ is the cyclic subgroup generated by $\alpha$, $X$ is ${\bf H}^4/\Phi$, and $\overline \alpha$ is the unique simple closed geodesic in $X$. With $\widetilde Y$ as the geodesic axis for $\alpha$, the preimage of $\overline \alpha$ under the covering map $p: {\bf H}^4 \longrightarrow X$, we again denote the convexification of $\overline \alpha$ in $X$ by $\mathcal C = p(C_P(\widetilde Y))$.

\vspace{.1 in}

\begin{lemma}\label{lem:Rcalculation}
Let $\mathcal C$ be the union of polyhedra forming the convexification of \,$\overline \alpha$ in $X$ via the procedure mentioned above. Then any polyhedron $P_i \in \mathcal C$ must intersect $N = N_R(\overline \alpha)$, the $R$--neighborhood of $\overline \alpha$, where $R = \ln(2 + \sqrt{3})$.
\end{lemma}

\begin{proof}
As in the proof of Lemma \ref{lem:nottoofar} we let $\widetilde N$ be the preimage of $N$ in ${\bf D}^4$, which forms an $R$--neighborhood around $\widetilde Y$. Suppose $P$ is a polyhedron in our tessellation of ${\bf D}^4$ which does not intersect $\widetilde N$. We aim to show that a hyperplane in ${\bf D}^4$ containing one of the codimension--1 faces of $P$ must separate $P$ from $\widetilde Y$, demonstrating that $P \notin \widetilde {\mathcal{C}}$ and proving the lemma. 

Let $e$ be the minimal-dimensional face of $P$ containing its closest point to $\widetilde Y$ and set $k = $codim$(e)$. Again, $k \in [1, 4]$ is also the number of codimension--1 faces of $P$ that intersect $e$. Take a shortest geodesic $\overline{py}$ from $e$ to $\widetilde Y$ that intersects $\widetilde Y$ at a point $y$ and $e$ at a point $p$. Then $\overline{py}$ is an orthogeodesic with $\ell(\overline{py}) \geq R$ since $P$ does not intersect the $R$--neighborhood of $\widetilde Y$. We let $j$ be a hyperplane through $y$ that is perpendicular to $\overline{py}$, which separates $e$ from $\widetilde Y$. Again $j$ necessarily contains $\widetilde Y$.

We note that the proof in the cases $k = 3, 2, 1$ are exactly the proofs for $k = 3, 2, 1$, respectively, of Lemma \ref{lem:nottoofar}. We therefore consider the case where $e$ is a vertex of $P$ and $k = 4$. We begin by sending $e = p$ to the origin of ${\bf D}^4$ via isometries. Recall that $L_{x_i}$ is the hyperplane of ${\bf D}^4$ obtained by restricting the $x_i$ coordinate to zero for $i = 1, 2, 3, 4$. Since $P$ is an all right polyhedron, the four codimension--1 faces of $P$ that intersect $e$ necessarily lie in four hyperplanes $L_1$, $L_2$, $L_3$, and $L_4$ that are the isometric image of $L_{x_1}$, $L_{x_2}$, $L_{x_3}$, and $L_{x_4}$. Letting $\partial L_i = \partial {\bf D}^3 \cap L_i$, we see that $\partial L_1$, $\partial L_2$, $\partial L_3$, and $\partial L_4$ form an all right-angled spherical tetrahedron in $\partial {\bf D}^4$. We are, therefore, interested in the threshold, $R$, of the distance between $P$ and $\widetilde Y$ so that $\partial j$ can be inscribed in an all right spherical tetrahedron formed by the boundaries of four pairwise orthogonal hyperplanes. Then if $d(P, \widetilde Y) > R$, at least one of the four two-spheres $\partial L_i$ cannot intersect $\partial j$, and thus, one of the four hyperplanes $L_i$ must separate $P$ from $\widetilde Y$. 

We begin by calculating the radius $r$ of a sphere inscribed in such an all right tetrahedron as shown in Figure \ref{fig:tetra}. As indicated by the figure, we calculate the radius $r$ using a spherical triangle formed by a midpoint of an edge of the tetrahedron, $A$, a point of tangency of the inscribed sphere,$M$, and the center $O$ of the sphere.

\begin{figure}[h]
\centering
\begin{overpic}[trim = .75in 3.5in 1.55in 2.35in, clip=true, totalheight=0.23\textheight]{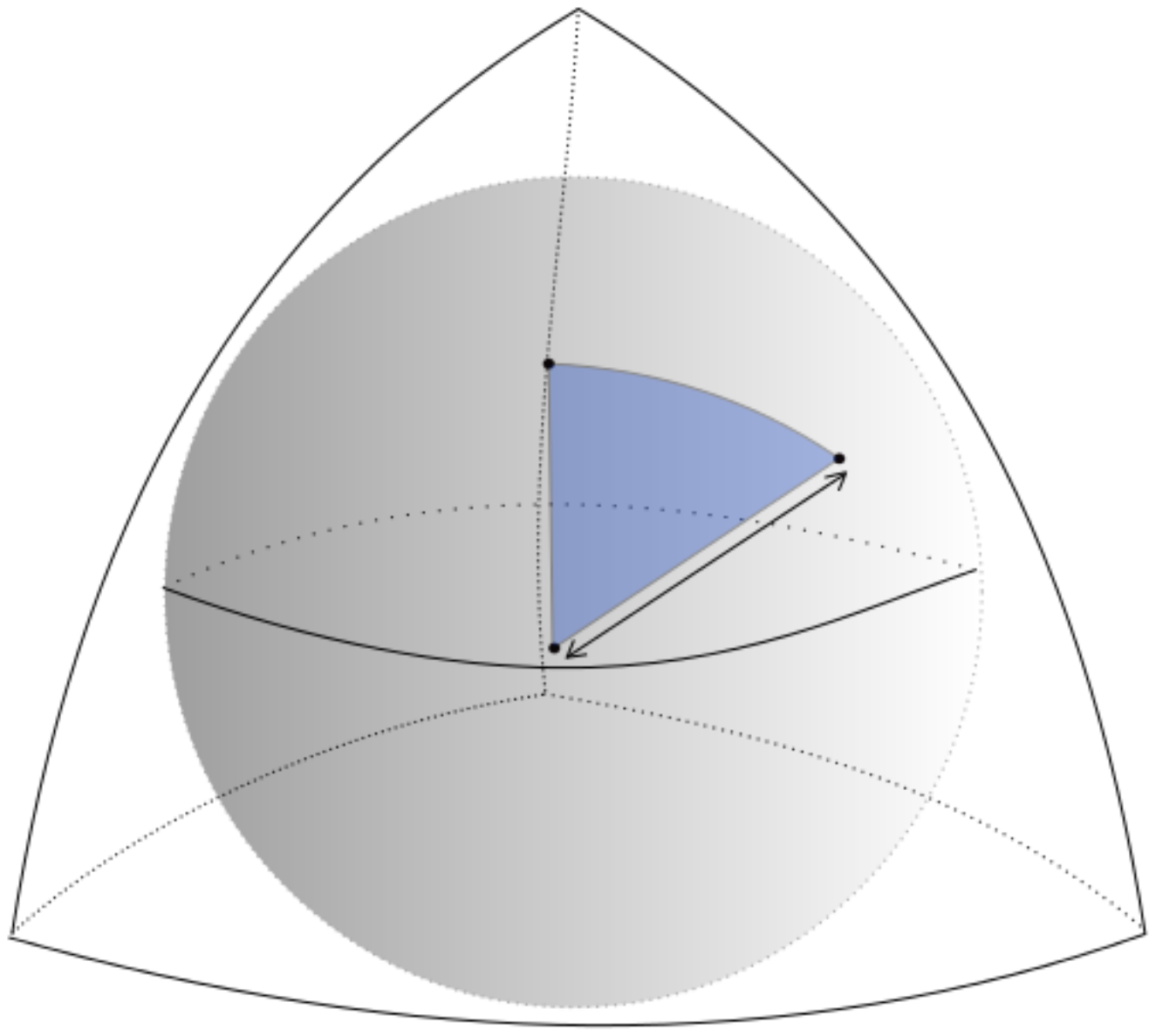}
\put(45,52){\scriptsize{$A$}}
\put(44,30){\scriptsize{$O$}}
\put(70,45){\scriptsize{$M$}}
\put(60,32){\scriptsize{$r$}}
\end{overpic}
\caption{}\label{fig:tetra}
\end{figure}

Note that the angle at $A$ is $\frac{\pi}{4}$, the angle at $M$ is $\frac{\pi}{2}$, and the length of the edge between $M$ and $A$ is the radius of a circle inscribed in a spherical triangle whose edge lengths are $\frac{\pi}{2}$, which we calculated above to be $\cos^{-1}\left(\frac{\sqrt{2}}{\sqrt{3}}\right)$. An application of a spherical law of cosines gives that the angle at $O$ is $\cos^{-1}\left(\frac{1}{\sqrt{3}}\right)$, and a second application gives that  $r = \cos^{-1}\left(\frac{\sqrt{3}}{2}\right)$.

We now use the same cross sectional picture as in the proof of Lemma \ref{lem:nottoofar}, which is shown is Figure \ref{fig:crosssec}. The angle $\theta$ is again equal to $r$, where now $r = \cos^{-1}\left(\frac{\sqrt{3}}{2}\right)$. Therefore, $c = \sec(\theta)= \frac{2}{\sqrt{3}}$ and $q= \tan(\theta) = \frac{1}{\sqrt{3}}$. It follows that $y = c - q = \frac{1}{\sqrt{3}}$ and $R = d(0, \widetilde Y) = \ln\left(\frac{1 + y}{1 - y}\right)$, which by a simple calculation gives us $R = \ln(\sqrt{4} + \sqrt{3}) = \ln(2 + \sqrt{3})$.

 Taking the maximum $R$ over the four cases gives $R = \ln(2 + \sqrt{3})$.

\end{proof}

The proof of Theorem \ref{thm:rfmain2} now follows exactly as the proof of Theorem \ref{thm:rfmain}. Lemma \ref{lem:Rcalculation} tells us that $\mathcal C \subset N_{R+d_P}(\overline \alpha)$, where $R$ can be taken to be $\ln(2 + \sqrt{3})$. We take a subset of the preimage of $N_{R+d_P}(\overline \alpha)$ in ${{\bf H}^4}$ that is isometric to $N_{R+d_P}(\overline{\alpha})$ and which forms a region like $\Omega$ from Lemma \ref{lem:volume4}, where $b = R + d_P$. The lemma then implies that \[\text{ Vol } (\mathcal C) <  \text{ Vol }(N_{R+d_P}(\overline \alpha)) = \frac{4}{3} \pi \sinh^3(R + d_P) \, \ell_\rho(\alpha) = \frac{4}{3} \pi \sinh^3(\ln(2 + \sqrt{3}) + d_P) \, \ell_\rho(\alpha).\]

\begin{theorem}\label{thm:rfmain2}
Let $(M, \rho)$ be a hyperbolic manifold that admits a totally geodesic immersion to a compact, right-angled Coxeter orbifold, $\mathcal{O}_P$, of dimension 4.Then for any $\alpha \in \pi_1(M)-\{1\}$, there exists a subgroup $H'$ of $\pi_1(M)$ such that $\alpha \notin H'$, and the index of $H'$ is bounded above by \[ \dfrac{8\pi }{3V_P}\sinh^3(\ln(2 + \sqrt{3}) + d_P) \,\ell_\rho(\alpha),\] where $\ell_\rho(\alpha)$ is the length of the unique geodesic representative of $\alpha$, and where $d_P$ and $V_P$ are the diameter and volume of $P$, respectively.

\end{theorem}

\begin{proof}
We know that Vol$(\mathcal C) < \frac{4}{3} \pi \sinh^3(\ln(2 + \sqrt{3}) + d_P) \, \ell_\rho(\alpha)$ so if $\mathcal C$ consists of $k$ polyhedra, \[k < \dfrac{4\pi}{3V_P} \sinh^3(\ln(2 + \sqrt{3}) + d_P) \, \ell_\rho(\alpha).\]

\noindent We form the convex set $\mathcal C' = \overline{\mathcal C} \cup \overline{\mathcal{C}_1}$ as in the proof of Theorem \ref{thm:rfmain}. Thus, $C'$ is the convex union of $2k$ copies of $P$ containing one endpoint of a lift $\widetilde \alpha$ of $\alpha$ to ${\bf H}^4$ in its interior. 

Let $H$ be the group of isometries of ${\bf H}^4$ generated by reflections in the sides of $\mathcal C'$. Then $H < \Gamma_P$, and $\mathcal C'$ is a fundamental domain for the action of $H$ on ${\bf H}^4$ by the Poincar\'{e} Polyhedron Theorem. Thus, the image of $\widetilde \alpha$ is not a loop in ${\bf H}^4/H$, and $\alpha \notin H$. Now, let $H' = H \cap \pi_1(M)$. Then, $\alpha \notin H'$ and $[\pi_1(M)\!:\!H'] \leq [\Gamma_P\!:\!H] = 2k$. The result follows. 

\end{proof}

\noindent An immediate corollary of Theorems \ref{thm:rfmain} and \ref{thm:rfmain2} is:

\begin{corollary}\label{cor:linear}
Let $(M, \rho)$ be a hyperbolic manifold admitting a totally geodesic immersion to a compact, right-angled Coxeter orbifold of dimension 3 or 4. Then, the geodesic residual finiteness growth is at most linear. That is to say, $\DM_{M, \rho} \preceq n$.
\end{corollary}

\section{Extension to Manifolds that Virtually Immerse into Compact Reflection Orbifolds}\label{sec:virtually}
In this section we extend Corollary~\ref{cor:linear} to all hyperbolic manifolds that virtually admit a totally geodesic immersion to a compact, right-angled Coxeter orbifold of dimension 3 or 4. The key fact is the following lemma, which is the analog of \cite[Lemma~2.2]{BHP} for geodesic residual finiteness growth functions (rather than the usual residual finiteness growth functions, calculated with respect to word length).

\begin{lemma}\label{lem:boundextends}
Let $(M, \rho)$ be a hyperbolic $n$--manifold and let $K \leq \pi_1(M)$ be a finite index subgroup with $[\pi_1(M) : K] = C$. Let $(M', \rho')$ be the cover of $M$ of index $C$ corresponding to the subgroup $K$. Then the geodesic residual finiteness function for $(M,\rho)$ is bounded by that of $(M', \rho')$. That is to say, $\DM_{M, \rho} \leq C\cdot \DM_{M', \rho'}$ and hence $\DM_{M, \rho} \preceq \DM_{M', \rho'}$.\end{lemma}

\begin{proof}
For an element $\alpha \in \pi_1(M)$,  we see that \[\left \{H \leq \pi_1(M) : \alpha \notin H \right\}  \supseteq  \left \{ K' \leq K \leq \pi_1(M) : \alpha \notin K' \right\}.\] Therefore, 
\begin{align} 
D_{\pi_1(M)}(\alpha) &= \min \left \{ [ \pi_1(M) : H] : \alpha \notin H, H \leq \pi_1(M) \right\} \nonumber \\
& \leq \min \left \{ [ \pi_1(M) : K'] : \alpha \notin K', K' \leq K \leq \pi_1(M) \right\} \nonumber \\
&= C  \min \left \{ [K : K'] : \alpha \notin K' , K' \leq K \right \}= C \cdot D_{K}(\alpha), \nonumber
\end{align}

\vspace{.1 in}

\noindent where we set $D_{K}(\alpha) = 1$ if $\alpha \notin K$. The equality above comes from the fact that $[ \pi_1(M) : K'] = [\pi_1(M) : K ] [K : K'] = C [K: K']$. 

Next, we claim that 
\begin{align}
\DM_{M, \rho}(n) = &\max \left \{D_{\pi_1(M)}(\alpha) : \alpha \in \pi_1(M) - \{1\}, \ell_{\rho}(\alpha) \leq n\right\} \nonumber \\
\leq &\max \left \{ C \cdot D_{K}(\beta) : \beta \in K -\{1\}, \ell_{\rho'}(\beta) \leq n  \right \} = C \cdot \DM_{M', \rho'}. \nonumber
\end{align}

\noindent First, observe that $D_{\pi_1(M)}(\alpha) \leq C$ for all $\alpha \notin K$. Thus, if the maximum value for $D_{\pi_1(M)}(\alpha)$ on the set $\left\{\alpha \in \pi_1(M) -\{1\} : \ell_\rho(\alpha) \leq n\right\}$ is achieved by an element $\alpha \notin K$, the above inequality follows from the fact that $\DM_{M', \rho'} \geq 1$.  Additionally, we have that for all $\alpha \in K \leq \pi_1(M)$, $\ell_{\rho}(\alpha) = \ell_{\rho'}(\alpha)$. If the maximum value of $D_{\pi_1(M)}(\alpha)$ on the set $\left\{\alpha \in \pi_1(M) -\{1\} : \ell_\rho(\alpha) \leq n\right\}$ is achieved by an element $\alpha \in K \leq \pi_1(M)$, then $\alpha \in \left\{\beta \in K -\{1\} : \ell_\rho'(\beta) \leq n\right\}$. Therefore, the maximum value of $C \cdot D_{K}(\beta)$ on the set $ \left\{\beta \in K -\{1\} : \ell_\rho'(\beta) \leq n\right\}$ is at least as big as $C \cdot D_{K}(\alpha) = D_{\pi_1(M)}(\alpha)$, and the inequality follows.

\end{proof}

\noindent Corollary \ref{cor:linear}, together with Lemma \ref{lem:boundextends}, gives the following:

\begin{corollary}\label{cor:virtually}
If $(M, \rho)$ is a hyperbolic manifold that virtually admits a totally geodesic immersion to a compact, right-angled Coxeter orbifold of dimension 3 or 4, then $\DM_{M, \rho} \preceq n$.
\end{corollary}

\section{Comparison of residual finiteness growth and geodesic residual finiteness growth}\label{sec:comparison}

\subsection{Equivalence of growths in the compact setting:}

Recall that the Svarc-Milnor Lemma \cite[P. 140]{BridsonHaefliger} tells us that for a compact hyperbolic manifold $(M, \rho)$ with universal cover $\widetilde M$, the map $\phi: \pi_1(M) \longrightarrow \widetilde{M}$ defined by $\phi(g) = g\cdot x$, for a fixed but arbitrary point $x \in \widetilde M$, is a quasi-isometry. We note that the compactness condition is necessary.

The universal cover $\widetilde M$ of $(M, \rho)$ can be embedded isometrically in ${\bf H}^n$ and the action of $\pi_1(M)$ on $\widetilde M$ extends naturally, so that $\pi_1(M)$ acts freely and properly discontinuously on ${\bf H}^n$. Of course when $M$ is closed, $\widetilde M$ is all of ${\bf H}^n$. Let $\alpha \in \pi_1(M) -\{1\}$ and recall that for a hyperbolic isometry $\alpha$ of ${\bf H}^n$, the distance between a point $z \in {\bf H}^n$ and its translate $\alpha \cdot z$ is minimized on the geodesic axis, $L_\alpha$, for $\alpha$. For a point $z_0$ on $L_\alpha$, the distance between $z_0$ and $\alpha \cdot z_0$ is the length of the unique geodesic representing $\alpha$ in $(M, \rho)$, which is often called the translation length of isometry $\alpha$. 

The following lemma shows that the residual finiteness growth and the geodesic residual finiteness growth for compact hyperbolic manifolds are the same:

\begin{lemma}\label{lem:equivalent}
Let $(M, \rho)$ be a compact hyperbolic manifold, and let $\DM_{{\pi_1(M)},\mathcal S}$ and $\DM_{M, \rho}$ be the residual finiteness function for $\pi_1(M)$ and the geodesic residual finiteness function for $M$, respectively. Then,  $\DM_{{\pi_1(M)},\mathcal S} \simeq \DM_{M, \rho}$. 
\end{lemma}

\begin{proof}

Fix any finite generating set $\mathcal{S}$ of $\pi_1(M)$ and a point $x \in \widetilde M$. We first aim to show that there exists a constant $C > 0$ such that for all $\alpha \in \pi_1(M) - \{1\}$ with $\|\alpha\|_{\mathcal S} \leq n$, the geodesic/translation length of $\alpha$ satisfies $\ell_\rho(\alpha) \leq C \cdot n$. If we can find such a constant, then

\begin{align}
\DM_{{\pi_1(M)},\mathcal S}(n) = &\max \left \{D_{\pi_1(M)}(\alpha) :\alpha \in \pi_1(M) - \{1\}, \|\alpha\|_{\mathcal S} \leq n  \right \} \nonumber \\
\leq &\max \left \{D_{\pi_1(M)}(\alpha) : \alpha \in \pi_1(M) - \{1\}, \ell_{\rho}(\alpha) \leq C \cdot n \right\} = \DM_{M, \rho}(C\cdot n), \nonumber
\end{align}

\noindent and we have $\DM_{{\pi_1(M)},\mathcal S} \preceq \DM_{M, \rho}$.

From the Svarc-Milnor Lemma, there exist constants $a, b >0$ such that 

\begin{equation}\label{eq:Svarc}
 \frac{1}{a} \cdot \|\alpha\|_{\mathcal S} - b \leq d(x, \alpha \cdot x) \leq a \cdot \|\alpha\|_{\mathcal S} +b
 \end{equation}

\noindent  for all $\alpha \in \pi_1(M)$, where $d$ measures distance in ${\bf H}^n$. In particular, if $\|\alpha\|_{\mathcal S} \leq n$, then $d(x, \alpha \cdot x) \leq C \cdot n$ for a constant $C>0$, depending only on $a$ and $b$. By the comments above on translation length, we have that $\ell_\rho(\alpha) \leq d(x, \alpha \cdot x) \leq C \cdot n$, which is the desired inequality.

Next, we demonstrate that $\DM_{M, \rho} \preceq \DM_{{\pi_1(M)},\mathcal S}$. With $x \in \widetilde M$ our fixed point from above, there exists a constant $D > 0$ such that every point $z \in \widetilde M$ is within distance $D$ of $h \cdot x$ for some $h \in \pi_1(M)$. 

For $\alpha \in \pi_1(M) -\{1\}$, let $g \cdot x$ be a point within distance $D$ of the geodesic axis, $L_\alpha$, for $\alpha$. Then $x$ is within distance $D$ of the geodesic axis, $L_{ g\alpha g^{-1}}$, for $g \alpha g^{-1}$. Let $y$ be the point on $L_{g \alpha g^{-1}}$ closest to $x$, so that $$d(x, y) = d( g \alpha g^{-1} \cdot x, g \alpha g^{-1} \cdot y) < D.$$

\noindent Then by the triangle inequality, 
\begin{equation}\label{eq:triangle}
d(x, g \alpha g^{-1}\cdot x) \leq 2D + d(y, g \alpha g^{-1} \cdot y) = 2D + \ell_\rho(g \alpha g^{-1}).
\end{equation}

Combining equations \ref{eq:Svarc} and \ref{eq:triangle}, we conclude that for every $\alpha \in \pi_1(M) -\{1\}$, there exists $g \in \pi_1(M)$ such that 

\begin{equation}\label{eq:lengths}
\|g \alpha g^{-1}\|_{\mathcal S} \leq C' \cdot \ell_\rho(g \alpha g^{-1}),
\end{equation}

\noindent for a constant $C' >0$ that does not depend on $\alpha$. The two key facts that we will need to complete the proof of the lemma are that $\ell(\alpha) = \ell(g \alpha g^{-1})$ and that $D_{\pi_1(M)}(\alpha) = D_{\pi_1(M)}(g \alpha g^{-1})$ for all $g \in \pi_1(M)$. Using the fact that $D_{\pi_1(M)}(\alpha) = D_{\pi_1(M)}(g \alpha g^{-1})$, we note that we can redefine the residual finiteness function by letting $\DM_{{\pi_1(M)},\mathcal S}(n)$ be the maximum value of $D_{\pi_1(M)}(\alpha)$ on the set $$ A_n = \left \{\alpha \in \pi_1(M) - \{1\}: \|g\alpha g^{-1}\|_{\mathcal S} \leq n \text{ for some }g \in \pi_1(M) \right \}.$$

Now, if $\alpha \in \pi_1(M) -\{1\}$ with $\ell_\rho(\alpha) \leq n$, then $\ell_\rho(g \alpha g^{-1}) \leq n$. Equation \ref{eq:lengths} gives us $\|g \alpha g^{-1}\|_{\mathcal S} \leq C' n$, so that $\alpha \in A_{C'n}$. Therefore,

\begin{align}
\DM_{M, \rho}(n) = &\max \left \{D_{\pi_1(M)}(\alpha) : \alpha \in \pi_1(M) - \{1\}, \ell_{\rho}(\alpha) \leq n\right\} \nonumber \\
\leq &\max \left \{D_{\pi_1(M)}(\alpha) :\alpha \in \pi_1(M) - \{1\}, \|g\alpha g^{-1}\|_{\mathcal S} \leq C' \cdot n \text{ for some } g \in \pi_1(M)  \right \} \nonumber \\ = &\DM_{{\pi_1(M)},\mathcal S}(C'\cdot n), \nonumber
\end{align}

\noindent We conclude that $\DM_{{\pi_1(M)},\mathcal S} \simeq \DM_{M, \rho}$.

\end{proof}

\subsection{Relationship between growths in the non-compact setting:}
The Svarc-Milnor Lemma plays a crucial role in the proof of Lemma \ref{lem:equivalent}, in particular for showing that $\DM_{{\pi_1(M)},\mathcal S} \preceq \DM_{M, \rho}$, and the compactness condition on $(M, \rho)$ is a necessary hypothesis of the lemma. Of course, this observation leads us the question whether the geodesic residual finiteness growth and the residual finiteness growth are the same for a non-compact, finite volume hyperbolic manifold. The work of Basmajian \cite{Basmajian} examining closed geodesics on non-compact, finite area hyperbolic surfaces suggests that the two growths may differ. Example \ref{example} below uses a sequence of curves that Basmajian studies in \cite{Basmajian} to demonstrate the fact that the geodesic length of an element $\alpha \in \pi_1(\Sigma)$ can be logarithmic in its word length when $\Sigma$ is a non-compact, finite area hyperbolic surface. In fact, the curves can be used to demonstrate that the geodesic length can be logarithmic in the \emph{cyclically reduced word length} $\| \alpha \|^c_{\mathcal S}$, which is the minimal word length over all conjugates of $\alpha$. This relationship suggests that it is strictly harder to establish linear bounds on geodesic residual finiteness growth for non-compact, finite volume hyperbolic manifolds since it would require separating elements from subgroups whose indexes are logarithmic in their cyclically reduced word lengths.

\begin{exmp}\label{example}
Let $\Sigma =\Sigma_{0,3}$ be the three-punctured sphere. Then $\pi_1(\Sigma) = \langle a, b \rangle$ is a free group on two generators, and we fix the obvious generating set $\mathcal{S} = \{a, b, a^{-1}, b^{-1}\}$. We also fix a hyperbolic structure $\rho$ on $\Sigma$ by choosing as a fundamental domain the region in ${\bf H}^2$ bounded by the vertical lines with real part -1 and 1, and the two semi-circles orthogonal to $\partial {\bf H}^2$ with endpoints -1, 0 and 0, 1. Then, we can identify $a$ with the parabolic isometry of ${\bf H}^2$ associated to the matrix $\begin{pmatrix}
1 & 2 \\ 0 & 1
\end{pmatrix}$ and $b$ with the matrix $\begin{pmatrix}
1 & 0 \\ 2 & 1
\end{pmatrix}$. Consider the element $\gamma_n = a b^n \in \pi_1(\Sigma)$. Then $\gamma_n$ has (cyclically reduced) word length $n +1$, and we see that the matrix representing $\gamma_n$ has trace $2 +4n$. Letting $\ell = \ell_\rho(\gamma_n)$, we have that $\cosh(\ell/2) = 1+2n$, and thus, the geodesic length of $\gamma_n$ grows logarithmically in the (cyclically reduced) word length of $\gamma_n$.

\end{exmp}

It should also be noted that, by \cite{PotVin}, there exist non-compact, finite volume all right polyhedra in dimension at least up to 8. It would be interesting to explore geometric methods towards bounding residual finiteness for hyperbolic manifolds that admit a totally geodesic immersion to a non-compact reflection orbifold associated to such a polyhedron. It seems that the geometric methods used here would need to be modified in the non-compact case.

\end{document}